\documentclass[12pt]{amsart}
\usepackage{latexsym, wasysym}
\usepackage{amsthm}
\usepackage{amsmath}
\usepackage{amsfonts}
\usepackage{amssymb}
\usepackage[dvips]{graphicx}
\usepackage{xy}
\xyoption{all}

\usepackage[active]{srcltx}  

\newtheorem{theo}{Theorem}[section]

\newtheorem{propo}[theo]{Proposition}
\newtheorem{defi}[theo]{Definition}
\newtheorem{coro}[theo]{Corollary}
\newtheorem{rem}[theo]{Remark}
\newtheorem{exam}[theo]{Example}

\theoremstyle{definition}

\theoremstyle{remark}

\newcommand\colim{\mathop{\rm colim}\limits}

\newcommand\Ind{\operatorname{Ind}}
\newcommand\Inj{\operatorname{Inj}}
\newcommand\Ort{\operatorname{Ort}}
\newcommand\PsPb{\operatorname{PsPb}}
\newcommand\Eq{\operatorname{Eq}}

\newcommand\Ins{\operatorname{Ins}}

\newcommand\pres{\operatorname{pres}}

\newcommand\op{\operatorname{op}}

\newcommand\id{\operatorname{id}}
\newcommand\Id{\operatorname{Id}}

\newcommand\Set{\operatorname{\bf Set}}

\newcommand\Top{\operatorname{\bf Top}}

\newcommand\cof{\operatorname{cof}}

\newcommand\ca{\mathcal {A}}
\newcommand\cb{\mathcal {B}}

\newcommand\cc{\mathcal {C}}
\newcommand\cd{\mathcal {D}}

\newcommand\ce{\mathcal {E}}

\newcommand\ck{\mathcal {K}}
\newcommand\cl{\mathcal {L}}
\newcommand\cm{\mathcal {M}}
\newcommand\crr{\mathcal {R}}

\newcommand\ct{\mathcal {T}}
\newcommand\cp{\mathcal {P}}
\newcommand\co{\mathcal {O}}

\newcommand\cx{\mathcal {X}}

\date{December 27, 2011}
  
\begin{document}
\title[Class-locally presentable categories]
{Class-locally presentable and class-accessible categories}
\author[B. Chorny and J. Rosick\'{y}]
{B. Chorny and J. Rosick\'{y}$^*$}
\thanks{ $^*$ Supported by MSM 0021622409 and GA\v CR 201/11/0528. The hospitality of the Australian National University 
is gratefully acknowledged.} 
\address{\newline B. Chorny\newline
Department of Mathematics\newline
University of Haifa at Oranim\newline
Tivon, Israel\newline
chorny@math.haifa.ac.il
\newline\newline 
J. Rosick\'{y}\newline
Department of Mathematics and Statistics\newline
Masaryk University, Faculty of Sciences\newline
Brno, Czech Republic\newline
rosicky@math.muni.cz
}

\begin{abstract}
We generalize the concepts of locally presentable and accessible categories. Our framework includes such categories as small presheaves over 
large categories and ind-categories. This generalization is intended for applications in the abstract homotopy theory.  
\end{abstract}

\maketitle
 
\section{Introduction}
The concept of a locally presentable category was introduced by Gabriel and Ulmer \cite{GU}. It was further generalized by Makkai and Par\'e 
who introduced accessible categories in the monograph \cite{MP} which convincingly demonstrated the importance of this notion. Since then 
locally presentable and accessible categories found numerous applications in algebra and, most prominently, in homotopy theory, where the concept 
of a locally presentable category was adapted by J. Smith as a foundation for his theory of combinatorial model categories 
\cite{Beke,D-universal,D}.

In the passed decade there were several interesting examples of non-combinatorial model structures constructed on non-locally presentable categories 
\cite{BCR, CD, CI, I}. The goal of our work is to extend the notions of the locally presentable and accessible categories, so that it could serve 
as a categorical foundation for an appropriate generalization of J.~Smith's theory. Such generalization will appear in the companion article \cite{CR}. 
 
The definition of an accessible category consists of a combination of a completeness condition and a smallness condition.  The smallness condition 
demands the existence of a set $\ca$ of $\lambda$-presentable objects such that each object is a $\lambda$-filtered colimit of objects of this set.
In our paper, we are dropping the assumption that $\ca$ is a set and we call the resulting concept a class-accessible category. This is not 
a new idea. Long before the appearance of \cite{MP}, this concept was introduced in \cite{BH} under the name of a $\lambda$-algebroidal category. 
The main disadvantage of $\ca$ being a class is that images of its objects by a functor can have arbitrarily large presentation ranks. 
Nevertheless, we will show that surprisingly many results about accessible categories can be generalized to the class-accessible setting. 
In particular, class-accessible categories are closed under lax limits. Furthermore, there is a satisfactory theory of injectivity and
weak factorization systems in class-locally presentable categories which is a starting point for systematic applications in homotopy theory 
(see \cite{CR})
leaning on several existing results in this direction (see \cite{BCR}, \cite{C}, \cite{C1} and \cite{CD}). 

The main example of a class-accessible category which is not accessible is the category of small presheaves on a large category. The omnipresence 
of this category was the main motivation for our work. Since orthogonality and factorization systems behave even better than injectivity and weak
factorization systems, a way is open for dealing with small sheaves. An early work in this direction is \cite{W}.  

We have to distinguish between sets and classes. This could be formalized by saying that we are working in the G\" odel-Bernays
set theory. Each set is a class, class which is not a set is called proper. A category consists of a class of objects but $\hom(A,B)$ 
are sets for every pair of objects $A,B$. Sometimes, such categories are called locally small (which forces us to change the terminology 
introduced by topologists \cite{Dr} for classes of morphisms satisfying the cosolution set condition --- we call them  
cone-coreflective in this paper). A category is small if it has a set of objects. 

\section{Class-accessible categories}
Let us recall that a category $\ck$ is called $\lambda$-\textit{accessible}, where $\lambda$ is
a regular cardinal, provided that
\begin{enumerate}
\item[(1)] $\ck$ has $\lambda$-filtered colimits,
\item[(2)] $\ck$ has a set $\ca$ of $\lambda$-presentable objects such that every object of $\ck$ is a 
$\lambda$-filtered colimit of objects from $\ca$.
\end{enumerate}

An object $K$ in $\ck$ is called $\lambda$-\textit{presentable} if its hom-functor 
$$
\hom(K,-) \colon \ck\to\Set
$$ 
preserves $\lambda$-filtered colimits. A category is \textit{accessible} if it is $\lambda$-accessible for some regular 
cardinal $\lambda$. A cocomplete ($\lambda$-)accessible category is called \textit{locally ($\lambda$-)presentable}. All 
needed facts about locally presentable and accessible categories can be found in \cite{AR} or \cite{MP}.  

\begin{defi}\label{def2.1}
{\em
A category $\ck$ is called \textit{class-$\lambda$-accessible}, where $\lambda$ is
a regular cardinal, provided that
\begin{enumerate}
\item[(1)] $\ck$ has $\lambda$-filtered colimits,
\item[(2)] $\ck$ has a class $\ca$ of $\lambda$-presentable objects such that every object of $\ck$ is a 
$\lambda$-filtered colimit of objects from $\ca$.
\end{enumerate}

A category is \textit{class-accessible} if it is class-$\lambda$-accessible for some regular cardinal $\lambda$. A complete and cocomplete 
class-$\lambda$-accessible category is called \textit{class-locally $\lambda$-presentable}. A category is \textit{class-locally presentable} 
if it is  class-locally $\lambda$-presentable for some regular cardinal $\lambda$.

Finally, a category $\ck$ is called \textit{class-preaccessible} if it satisfies (2) for some regular cardinal $\lambda$.
}
\end{defi}

\begin{exam}\label{ex2.2}
{
\em
(1) Each $\lambda$-accessible category is class-$\lambda$-accessible. Sin\-ce each locally presentable category 
is complete, each locally $\lambda$-pre\-sen\-tab\-le category is class-locally $\lambda$-presentable. 

(2) Given a category $\ca$, $\cp(\ca)$ will denote the category of \textit{small presheaves} on $\ca$. These are
functors $\ca^{\op}\to\Set$ which are small colimits of hom-functors. For a small category $\ca$, we have
$\cp(\ca)=\Set^{\ca^{\op}}$.
The category $\cp(\ca)$ is always class-finitely-accessible (= class-$\omega$-acce\-ssi\-ble) because each small presheaf 
is a small filtered colimit of finite colimits of hom-functors and the latter are finitely presentable. This relies on a general 
fact that arbitrary colimits may be expressed as filtered colimits of finite colimits. 
$\cp(\ca)$ is always cocomplete but not necessarily complete. For instance, it does not have a terminal object 
in the case when $\ca$ is a large discrete category (it means that it has a proper class of objects 
and the only morphisms are the identities). This explains why we added completeness into the definition of a class-locally 
presentable category.

(3) The category $\Top$ of topological spaces is not class-locally presentable. The reason is that the only
presentable objects are discrete spaces.

(4) Given a category $\ca$, let $\Ind(\ca)$ be the full subcategory of $\cp(\ca)$ consisting of small filtered
colimits of hom-functors. This construction was introduced by Grothendieck (see \cite{AGV}) and $\Ind(\ca)$ 
is always class-finitely-accessible. In fact, each class-finitely-accessible category 
$\ck$ is equivalent to $\Ind(\ca)$ for $\ca$ being the full subcategory of $\ck$ consisting of finitely presentable 
objects. The proof is the same as in the case of finitely accessible categories.

This can be generalized to each regular cardinal $\lambda$ by introducing $\Ind_\lambda(\ca)$ as the full
subcategory of $\cp(\ca)$ consisting of small $\lambda$-filtered colimits of hom-functors. Then $\Ind_\lambda(\ca)$ 
is always class-$\lambda$-accessible and, conversely, each class-$\lambda$-accessible category $\ck$ is equivalent
to $\Ind_\lambda(\pres_\lambda\ck)$ where $\pres_\lambda\ck$ is the full subcategory of $\ck$ consisting 
of $\lambda$-presentable objects. The proof is the same as in \cite{AR} 2.26.
}
\end{exam}

\begin{rem}\label{re2.3}
{\em
(1) For each category $\ca$, the Yoneda embedding yields the functor
$$
Y\colon\ca\to\cp(\ca)
$$
making $\cp(\ca)$ a free completion of $\ca$ under small colimits. Analogously,
$$
Y:\ca\to\Ind_\lambda(\ca)
$$
is a free completion of $\ca$ under small $\lambda$-filtered colimits (the second $Y$ is the codomain restriction of the first one). 
 
(2) The category $\Ind_\lambda(\ca)$ is cocomplete if and only if $\ca$ is $\lambda$-\textit{co\-com\-ple\-te} in the sense 
that it has $\lambda$-\textit{small colimits}, i.e., colimits of diagrams $D:\cd\to\ca$ such that the category $\cd$ has less 
than $\lambda$ morphisms. The proof is the same as in the case of accessible categories, i.e., when $\ca$ is small (see, 
e.g., \cite{AR}, 1.46). One can also proceed as follows.

(3) Given a class-$\lambda$-accessible category $\ck$, we can express the class $\pres_\lambda\ck$ as a union
of an increasing chain of small subcategories indexed by all ordinals
$$
\ca_0\subseteq\ca_1\subseteq\dots\ca_i\subseteq\dots
$$
This results in writing $\ck$ as a union of a chain of  $\lambda$-accessible categories
$$
\Ind_\lambda\ca_0\subseteq\Ind_\lambda\ca_1\subseteq\dots\Ind_\lambda\ca_i\subseteq\dots
$$
and functors preserving $\lambda$-filtered colimits and $\lambda$-presentable objects. If $\ck$ is class-locally $\lambda$-presentable, 
we can assume that the first chain consists of subcategories closed under $\lambda$-small colimits. Then the second chain consists of colimit 
preserving functors. This proves (2).

(4) Recall that, given a small full subcategory $\ca$ of a category $\ck$ and an object $K$ in $\ck$, the \textit{canonical diagram}
of $K$ (with respect to $\ca$) is the forgetful functor $D:\ca\downarrow K\to\ck$. Here, $\ca\downarrow K$ consists of all morphisms
$a:A\to K$ with $A$ in $\ca$ and $D$ sends $a$ to $A$. We say that $K$ is a \textit{canonical colimit} of $\ca$-objects if the family $a:A\to K$
form a colimit cocone from the canonical diagram. If $\ck$ is $\lambda$-accessible than each object of $\ck$ is a canonical colimit 
of its canonical diagram with respect to $\pres_\lambda\ck$ and this canonical diagram is $\lambda$-filtered (see \cite{AR}, 
Proposition 2.8).

Now, let $\ck$ be a class-$\lambda$-accessible category. A consequence of (3) is that, for each object $K$ of $\ck$, there is a small
full subcategory $\ca$ of $\pres_\lambda\ck$ such that $K$ is a canonical $\lambda$-filtered colimit of $\ca$-objects.

(5) The category $\ca\downarrow K$ used above is a special case of a general \textit{comma-category} $F_1\downarrow F_2$ where
$F_1:\ck_1\to\cl$ and $\ck_2\to\cl$ are functors: we take $F_1$ as the inclusion of $\ca$ to $\ck$ and $F_2$ the functor from 
the one-morphism category to $\ck$ with the value $K$. The category $F_1\downarrow F_2$ has morphisms $f:F_1K_1\to F_2K_2$ as
objects and morphisms $f\to f'$ are pairs of morphisms $k_1,k_2$ for which the square
$$
\xymatrix@=4pc{
F_1K_1 \ar [r]^{f} \ar [d]_{F_1k_1}& F_2K_2 \ar [d]^{F_2k_2}\\
F_1K_1'\ar [r]_{f'}& F_2K_2'
}
$$
commutes.

Another special case is the category of morphisms $\ck^\to=\Id\downarrow\Id$.

(6) Every locally presentable category is cowellpowered, which does not generalize to class-locally presentable ones. For example,
the ordered class $\ck$ of ordinals with the added largest element is class-locally $\omega$-presentable with isolated ordinals as 
$\omega$-presentable objects. But $\ck$ is not cowellpowered. Hence a class-locally presentable category does not need to be locally 
ranked in the sense of \cite{AHRT}. Thus Theorem III.7 there does not imply our \ref{th4.3}.
}
\end{rem}

We will now give a criterion for when $\Ind_\lambda(\ca)$ is class-locally presentable, if $\ca$ is $\lambda$-cocomplete, i.e., 
when $\Ind_{\lambda}(\ca)$ is complete. We will need the following concepts. 
A set $\cx$ of objects of a category $\ck$ is called \textit{weakly initial} if each object $K$ of $\ck$ admits 
a morphism $X\to K$ with $X\in\cx$. A category $\ck$ is called \textit{approximately complete} if, for each diagram
$D:\cd\to\ck$, the category of cones $K\to D$ over $D$ has a weakly initial set. 

\begin{theo}\label{th2.4}
Let $\ca$ be a $\lambda$-cocomplete category (where $\lambda$ is a regular cardinal). Then $\Ind_\lambda(\ca)$ is
complete if and only if $\ca$ is approximately complete.
\end{theo}
\begin{proof}
Following \cite{F} and \cite{R}, $\cp(\ca)$ is complete if and only if $\ca$ is approximately complete. Since
$\ca$ is $\lambda$-cocomplete, $\Ind_\lambda(\ca)$ is cocomplete (see \ref{re2.3} (2)). Since $\cp(\ca)$ is a free completion of $\ca$
under colimits, there is a colimit preserving functor $F\colon\cp(\ca)\to\Ind_\lambda(\ca)$ such that the composition
$$
\ca \xrightarrow{\quad  Y\quad} \cp(\ca)
             \xrightarrow{\quad F\quad} \Ind_\lambda(\ca)
$$
is naturally isomorphic to the Yoneda embedding. Consequently, $F$ is left adjoint to the inclusion of $\Ind_\lambda(\ca)$ 
into $\cp(\ca)$ (see the proof of \cite{AR}, 1.45). Thus $\Ind_\lambda(\ca)$ is a reflective full subcategory of $\cp(\ca)$.
Hence $\Ind_\lambda(\ca)$ is complete whenever $\ca$ is approximately complete.

Conversely, let $\Ind_\lambda(\ca)$ be complete and consider a diagram $D\colon\cd\to\ca$. We express its limit
$K$ in $\Ind_\lambda(\ca)$ as a filtered colimit $(k_e\colon A_e\to K)_{e\in\ce}$ of objects from $\ca$. Now, each cone 
$A\to D$ with $A\in\ca$, uniquely factorizes through the limit cone via the morphism $t\colon A\to K$. Since $t$ factorizes 
through some $k_e$, the cones $A_e\to D$ obtained by precomposing the limit cone with $k_e$,
$e\in\ce$ form a weakly initial set of cones over $D$. Thus $\ca$ is approximately complete.
\end{proof}

\begin{rem}\label{re2.5}
{\em
Following \ref{th2.4}, a class-$\lambda$-accessible category $\ck$ is class-locally $\lambda$-presentable
if and only if $\pres_\lambda\ck$ is $\lambda$-cocomplete and approximately complete. Following \cite{F} 
and \cite{R}, $\cp(\ca)$ is class-locally finitely-presentable iff $\ca$ is approximately complete.
}
\end{rem}

\begin{theo}\label{th2.6}
Let $\ck$ be a category and $\lambda$ a regular cardinal. Then $\ck$ is class-locally $\lambda$-presentable
if and only if it is equivalent to a full, reflective subcategory of $\cp(\ca)$ closed under $\lambda$-filtered
colimits for some approximately complete category $\ca$.
\end{theo}
\begin{proof}
Given a class-locally $\lambda$-presentable category $\ck$, we put $\ca=\pres_\lambda\ck$ and define 
the \textit{canonical functor}
$$
E:\ck\to\cp(\ca)
$$
by taking $E(K):\ca^{\op}\to\Set$ to be the restriction of $\hom(-,K)$ on $\ca$. Since
$$
E(\colim D)\cong\colim ED
$$
for each $\lambda$-filtered diagram $D:\cd\to\ck$, each $E(K)$ is a $\lambda$-filtered colimit of hom-functors
and thus belongs to $\cp(\ca)$. Moreover, the functor $E$ preserves $\lambda$-filtered colimits. Since $\cp(\ca)$ 
is a free completion of $\ca$ under colimits, there is a colimit preserving functor $F:\cp(\ca)\to\ck$ such that 
the composition
$$
\ck \xrightarrow{\quad  E\quad} \cp(\ca)
             \xrightarrow{\quad F\quad} \ck
$$
is naturally isomorphic to $\Id_\ck$. Moreover, $F$ is left adjoint to the inclusion of $\ck$ into $\cp(\ca)$. Thus 
$\ck$ is a reflective full subcategory of $\cp(\ca)$. Finally, following the proof of \ref{th2.4}, $\ca$
is approximately complete.

Conversely, let $\ca$ be an approximately complete category and $\ck$ a full reflective subcategory of $\cp(\ca)$
closed under $\lambda$-filtered colimits. Since the reflection $F:\cp(\ca)\to\ck$ is left adjoint to the inclusion
of $\ck$ in $\cp(\ca)$ which preserves $\lambda$-filtered colimits, $F$ preserves $\lambda$-presentable objects.
Since each object of $\cp(\ca)$ is a $\lambda$-filtered colimit of $\lambda$-presentable objects, $\ck$ has the same 
property. Thus $\ck$ is class-locally $\lambda$-presentable.  
\end{proof}

\begin{defi}\label{def2.7}
{
\em
A functor $F:\ck\to\cl$ is called \textit{class-$\lambda$-accessible} (where $\lambda$ is a regular cardinal)
if $\ck$ and $\cl$ are class-$\lambda$-accessible categories and $F$ preserves $\lambda$-filtered colimits.
A class-$\lambda$-accessible functor preserving $\lambda$-presentable objects is called \textit{strongly
class-$\lambda$-accessible}. 

$F$ is called \textit{(strongly) class-accessible} if it is (strongly) class-$\lambda$-acce\-ssi\-ble 
for some regular cardinal $\lambda$.
}
\end{defi}

The uniformization theorem \cite{MP}, 2.4.9 (see \cite{AR}, 2.19) implies that each accessible functor is strongly accessible.
Among others, this uses the fact that, given a set $\ca$ of objects of an accessible category $\ck$, there is a regular cardinal 
$\lambda$ such that each $A\in\ca$ is $\lambda$-presentable. This does not generalize to class-accessible case and we have to
distinguish between class-accessible and strongly class-accessible functors here. For example, the large discrete category $\cd$
is class-$\omega$-accessible and any functor from $\cd$ into a class-$\omega$-accessible category is class-$\omega$-accessible
but not always strongly class-$\omega$-accessible. 

\begin{rem}\label{re2.8}
{
\em
In the same way as for accessible categories, one can replace $\lambda$-filtered colimits in \ref{def2.1} by
$\lambda$-directed ones. Moreover, one can show that, given a regular cardinals $\lambda\vartriangleleft\mu$, 
each class-$\lambda$-accessible category $\ck$ is class-$\mu$-accessible. The argument (see \cite{AR}, 2.11) goes 
as follows. One writes $K\in\ck$ as a $\lambda$-directed colimit $(a_i:A_i\to K)_{i\in I}$ of $\lambda$-presentable
objects. Let $\hat{I}$ be the poset of all $\lambda$-directed subsets of $I$ of cardinality less then $\mu$
(ordered by inclusion). Due to $\lambda\vartriangleleft\mu$, $\hat{I}$ is $\mu$-directed. Colimits of $(A_i)_{i\in M}$,
$M\in\hat{I}$ are $\mu$-presentable and $K$ is their $\mu$-directed colimit.

If $K$ is $\mu$-presentable then $K$ is a retract of some $\lambda$-directed colimit of $(A_i)_{i\in M}$. Consequently,
a strongly class-$\lambda$-accessible functor $F:\ck\to\cl$ is strongly class-$\mu$-accessible.

Recall that the successor $\lambda^+$ of each cardinal $\lambda$ is always regular and $\lambda\vartriangleleft\lambda^+$.
}
\end{rem}

\section{Limits of class-accessible categories}

The fundamental discovery of \cite{MP} is that accessible categories are closed under constructions of limit type.
In particular, they are closed under pseudopullbacks. The distinction between pullbacks and pseudopullbacks
is that the latter use isomorphisms instead of identities (see the proof of \ref{prop3.1} below). We are
going to show that this generalizes to class-accessible categories. For a pseudopullback 
$$
\xymatrix@=4pc{
\cp \ar [r]^{\bar{F}} \ar [d]_{\bar{G}}& \cl \ar [d]^{G}\\
\ck\ar [r]_{F}& \cm
}
$$
we will use the notation $\cp=\PsPb(F,G)$.

\begin{propo}\label{prop3.1}
Let $\lambda$ be a regular cardinal and $F:\ck\to\cm$ and $G:\cl\to\cm$ strongly class-$\lambda$-accessible functors. 
Then their pseudopullback $\PsPb(F,G)$ is a class-$\lambda^+$-accessible category and $\bar{F},\bar{G}$ are strongly
class-$\lambda^+$-accessible functors.
\end{propo}
\begin{proof}
Let $F$ and $G$ be strongly class-$\lambda$-accessible. Objects of their pseudo\-pull\-back are 5-tuples $(K,L,M,f,g)$ where $K\in\ck$,
$L\in\cl$, $M\in\cm$ and $f:FK\to M$, $g:GL\to M$ are isomorphisms (morphisms are obvious). Since both $F$ and $G$ preserve 
$\lambda$-filtered colimits, $\PsPb(F,G)$ has $\lambda$-filtered colimits and $\bar{F}$ and $\bar{G}$ preserve them. It remains to show that 
each object $(K,L,M,f,g)$ from $\PsPb(F,G)$ is a $\lambda^+$-filtered colimit of objects from $\PsPb(F,G)$ which are $\lambda^+$-presentable 
in $\ck\times\cl$. Following \ref{re2.8}, $\ck,\cl$ and $\cm$ are class-$\lambda^+$-accessible and $F$ and $G$ are strongly 
class-$\lambda^+$-accessible. 

Following \ref{re2.3} (4), there is a small full subcategory $\ca_1$ of $\pres_{\lambda^+}\ck$, a small full subcategory
$\ca_2$ of $\pres_{\lambda^+}\cl$ and a small full subcategory $\ca_3$ of $\pres_{\lambda^+}\cm$, such that $K$ is a canonical 
$\lambda^+$-filtered colimit of $\ca_1$-objects, $L$ is a canonical $\lambda^+$-filtered colimit of $\ca_2$-objects and $M$ is 
a canonical $\lambda^+$-filtered colimit of $\ca_3$-objects. We can also assume that $\ca_1$, $\ca_2$ and $\ca_3$ are closed under
$\lambda^+$-small $\lambda$-filtered colimits. We will denote the canonical diagrams as $C:\cc\to\ck$, $D:\cd\to\cl$ and $E:\ce\to\cm$
and their canonical colimits as $(k_c:Cc\to K)_{c\in\cc}$, $(l_d:Dd\to L)_{d\in\cd}$ and $(m_e:Ee\to M)_{e\in\ce}$ . Let $c_0\in\cc$,
$d_0\in\cd$ and $e_0\in\ce$. Since $FCc_0$ and $GDd_0$ are $\lambda^+$-presentable, there is $m_{01}:e_0\to e_1$ in $\ce$,
$f_0:FCc_0\to Ee_1$ and $g_0:GDd_0\to Ee_1$ such that $fF(k_{c_0})=m_{e_1}f_0$ and $gG(l_{d_0})=m_{e_1}g_0$. Analogously, there is 
$k_{01}:c_0\to c_1$ in $\cc$, $l_{01}:d_0\to d_1$ in $\cd$ and morphisms $f'_1:Ee_1\to FCc_1$, $g'_1:Ee_1\to GDd_1$ such that 
$f^{-1}m_{e_1}=F(k_{c_1})f'_1$, $f'_1f_0=FC(k_{01})$, $g^{-1}m_{e_1}=G(l_{d_1})g'_1$ and $g'_1g_0=GD(l_{01})$. There is 
$m_{12}:e_1\to e_2$ in $\ce$, $f_1:FCc_1\to Ee_2$ and $g_1:GDd_1\to Ee_2$ such that $fF(k_{c_1})=m_{e_2}f_1$, $f_1f'_1=E(m_{12})$,
$gG(l_{d_1})=m_{e_2}g_1$ and $g_1g'_1=E(m_{12})$. Continuing this procedure, we get chains $(k_{ij}:c_i\to c_j)_{i<j<\lambda}$, 
$(l_{ij}:d_i\to d_j)_{i<j<\lambda}$, $(m_{ij}:e_i\to e_j)_{i<j<\lambda}$ and morphisms $f_i:FCc_i\to Ee_{i+1}$, $g_i:GDd_i\to Ee_{i+1}$,
$f'_i:Ee_i\to FDd_i$, $g'_i:Ee_i\to GDd_i$ such that $f^{-1}m_{e_i}=F(k_{c_i})f'_i$, $f'_jf_i=FC(k_{ij})$, $g^{-1}m_{e_i}=G(l_{d_i})g'_i$ 
and $g'_jg_i=GD(l_{ij})$ for each $0<i<j<\lambda$. In limit steps, we take upper bounds (using the fact that $\cc$, $\cd$ and $\ce$ 
are $\lambda$-filtered). Let $K_\lambda=\colim Cc_i$, $L_\lambda=\colim Dd_i$ and $M_\lambda=\colim Ee_i$ where $i<\lambda$. We get morphisms 
$f_\lambda=\colim f_i:FK_\lambda\to M_\lambda$, $f'_\lambda=\colim f'_i:M_\lambda\to FK_\lambda$, $g_\lambda=\colim g_i:GL_\lambda\to M_\lambda$ 
and $g'_\lambda=\colim g'_i:M_\lambda\to GL_\lambda$ such that $f'_\lambda=(f_\lambda)^{-1}$ and $g'_\lambda=(g_\lambda)^{-1}$. Thus 
$(K_\lambda,L_\lambda,M_\lambda,f_\lambda,g_\lambda)$ belongs to $\PsPb(F,G)$. Since $K_\lambda$ belongs 
to $\ca_1$, $L_\lambda$ belongs to $\ca_2$ and $M_\lambda$ belongs to $\ca_3$, we have found a factorization of 
$$
(k_{c_0},l_{d_0}): (Cc_0,Dd_0)\to (K,L)
$$
through $(K_\lambda,L_\lambda,M_\lambda,f_\lambda,g_\lambda)\to(K,L,M,f,g)$. The consequence is that $(K,L,M,f,g)$ is a $\lambda^+$-filtered 
colimit of objects $(K_\lambda,L_\lambda,M_\lambda,f_\lambda,g_\lambda)$ which are $\lambda^+$-presentable in $\PsPb(F,G)$. Thus $\PsPb(F,G)$ is  
class-$\lambda^+$-accessible and $\bar{F}$ and $\bar{G}$ are strongly class-$\lambda^+$-accessible.
\end{proof}

\begin{rem}\label{re3.2}
{
\em
A functor $G:\ck\to\cl$ is \textit{transportable} if for every ob\-ject $L$ in $\cl$ and an isomorphism $g:GK\to L$
there exists a unique isomorphism $f:K\to K'$ such that $GK'=L$ and $Gf=g$.

If one of functors $F$ and $G$ in \ref{prop3.1} is transportable then, following \cite[5.1.1]{MP} , $\PsPb(F,G)$ is equivalent to the pullback
of $F$ and $G$.
}
\end{rem}
 
We follow the terminology of \cite{AR} and call a full subcategory $\cl$ of a category $\ck$ \textit{accessibly
embedded} if there is a regular cardinal $\lambda$ such that $\cl$ is closed under $\lambda$-filtered colimits
in $\ck$. Moreover, we say that $\cl$ is \textit{strongly accessibly embedded} if there is a regular cardinal $\lambda$
such that $\cl$ is closed under $\lambda$-filtered colimits in $\ck$ and the inclusion of $\cl$ to $\ck$ preserves 
$\lambda$-presentable objects.

\begin{coro}\label{cor3.3}
Let $\ck$ be a class-accessible category. An intersection of a set of class-accessible strongly accessibly embedded
subcategories of $\ck$ is a class-accessible strongly accessibly embedded subcategory of $\ck$.
\end{coro}
\begin{proof}
Let $\cl_i$, $i\in I$ be a set of class-accessible strongly accessibly embedded subcategories of $\ck$ and $\cl$ be
their intersection. Then $\cl$ is a multiple pullback of embeddings $\cl_i\to\ck$. Following \ref{re2.8} and \cite{AR}, 2.13 (6),
there is a regular cardinal $\lambda$ such that each inclusion $\cl_i\to\ck$ is strongly class-$\lambda$-accessible. 
In the usual way, this multiple pullback can be expressed as an equalizer of two functors between products. A product 
of class-accessible categories is class accessible and both functors are strongly class-$\lambda$-accessible.
The equalizer can be replaced by a pullback of strongly class-$\lambda$-accessible functors. Since this pullback is equivalent 
to their pseudopullback (cf. \cite{MP} 5.1.1), it follows from \ref{prop3.1} that $\cl$ is class-accessible strongly class-accessibly 
embedded subcategory of $\ck$.  
\end{proof}

\begin{rem}\label{re3.4}
{\em
(1) Given a class-accessible functor $F:\cl\to\ck$, the full subcategory of $\ck$ consisting of objects $F(L)$,
$L\in\cl$ is called the \textit{full image} of $F$.
 
(2) It is easy to see that, if a full image of a strongly class-accessible functor $\cl\to\ck$ is closed in $\ck$ 
under $\lambda$-filtered colimits for some regular cardinal $\lambda$, then it is class-accessible.
}
\end{rem}

\begin{coro}\label{cor3.5}
Let $F_1:\cl_1\to\ck$ and $F_2:\cl_2\to\ck$ be strongly class-accessible functors. Then the pseudopullback $\cp$ 
of their full images $F_1(\cl_1)\to\ck$ and $F_2(\cl_2)\to\ck$ is the full image of a strongly class-accessible functor
$\cm\to\ck$.
\end{coro}
\begin{proof}
It follows from \ref{prop3.1} in the same way as in \cite{R2}, 2.6.
\end{proof}

Recall that, given functors $F,G:\ck\to\cl$ and natural transformations $\varphi,\psi:F\to G$, the \textit{equifier} $\Eq(\varphi,\psi)$
is the full subcategory of $\ck$ consisting of all objects $K$ such that $\varphi_K=\psi_K$.
 
\begin{propo}\label{prop3.6}
Let $F,G:\ck\to\cl$ be strongly class-accessible functors and $\varphi,\psi:F\to G$ natural transformations.
Then $\Eq(\varphi,\psi)$ is a class-accessible category strongly accessibly embedded in $\ck$.
\end{propo}
\begin{proof}
Assume that $F$ and $G$ are strongly class-$\lambda$-accessible. The first observation is that, given a natural
transformation $\alpha:F\to G$ and a $\lambda$-filtered colimit $K=\colim K_d$, we have 
$\alpha_K=\colim\alpha_{K_d}$. As a consequence we get that $\Eq(\varphi,\psi)$ is closed in $\ck$ under
$\lambda$-filtered colimits.

We will show that each object $K\in\Eq(\varphi,\psi)$ is a $\lambda^+$-directed colimit of objects of 
$\Eq(\varphi,\psi)$ which are $\lambda^+$-presentable in $\ck$. This will yield that $\Eq(\varphi,\psi)$ is class-$\lambda^+$-accessible
and the inclusion of $\Eq(\varphi,\psi)$ into $\ck$ is strongly class-$\lambda^+$-accessible. We will proceed in a similar way as 
in the proof of \ref{prop3.1}. First, we know that $\ck$ is class-$\lambda^+$-accessible. Consider 
$K\in\Eq(\varphi,\psi)$ and take a small full subcategory $\ca$ of $\pres_{\lambda^+}\ca$ such that $K$ is a canonical
$\lambda^+$-filtered colimit of $\ca$-objects. We also assume that $\ca$ is closed in $\ck$ under $\lambda^+$-small
$\lambda$-filtered colimits. We denote the canonical diagram as $D:\cd\to\ck$ and its canonical colimit as
$(k_d:Dd\to K)_{d\in\cd}$. Let $d_0\in\cd$. Since 
$$
k_{d_0}\varphi_{Dd_0}=k_{d_0}\psi_{Dd_0},
$$
there is $k_{01}:d_0\to d_1\in\cd$ such that $D(k_{01})\varphi_{Dd_0}=D(k_{01})\psi_{Dd_0}$. Continuing this procedure, we get a chain 
$(k_{ij}:d_i\to d_j)_{i<j<\lambda}$, such that 
$$
D(k_{ij})\varphi_{Dd_i}=D(k_{ij})\psi_{Dd_i}
$$
for $i<j<\lambda$. In limit steps, we take upper bounds. Now, $K'=\colim Dd_i$, $i<\lambda$ belongs both to $\Eq(\varphi,\psi)$ 
and to $\ca$. Since we factorized $k_{d_0}$ through $K'$, $K$ is a $\lambda^+$-filtered colimit of objects from $\Eq(\varphi,\psi)$
which are $\lambda^+$-presentable in $\ck$.
\end{proof} 

Recall that, given functors $F,G:\ck\to\cl$, the \textit{inserter category} $\Ins(F,G)$ is the subcategory of the comma category $F\downarrow G$ 
(cf. \ref{re2.3} (4)) consisting of all objects $f:FK\to GK$ and all morphisms
$$
\xymatrix@=4pc{
FK \ar [r]^{f} \ar [d]_{Fk}& GK \ar [d]^{Gk}\\
FK'\ar [r]_{f'}& GK'
}
$$
The \textit{projection functor} $P:\Ins(F,G)\to\ck$ sends $f:FK\to GK$ to $K$. 
 
\begin{propo}\label{prop3.7}
Let $F,G:\ck\to\cl$ be strongly class-accessible functors. Then $\Ins(\varphi,\psi)$ is a class-accessible category and the projection
functor $P:\Ins(F,G)\to\ck$ is strongly class-accessible.
\end{propo}
\begin{proof}
Assume that $F$ and $G$ are strongly class-$\lambda$-accessible. It is easy to see that $\Ins(F,G)$ has $\lambda$-filtered colimits
preserved by $P$. We will show that each object $f:FK\to GK$ in $\Ins(F,G)$ is a $\lambda^+$-directed colimit of objects
$a:FA\to GA$ with $A$ $\lambda^+$-presentable in $\ck$. Since $a$ is then $\lambda^+$-presentable in $\Ins(F,G)$, we will get
that $\Ins(F,G)$ is class-$\lambda^+$-accessible and $P$ is strongly class-$\lambda^+$-accessible.

We will proceed in a similar way as in the proof of \ref{prop3.1}. First, we know that $\ck$ is class-$\lambda^+$-accessible. 
Consider $f:FK\to GK$ in $\Ins(F,G)$ and take a small full subcategory $\ca$ of $\pres_{\lambda^+}\ck$ such that $K$ is a canonical
$\lambda^+$-filtered colimit of $\ca$-objects. We also assume that $\ca$ is closed in $\ck$ under $\lambda^+$-small
$\lambda$-filtered colimits. We denote the canonical diagram as $D:\cd\to\ck$ and its canonical colimit as
$(k_d:Dd\to K)_{d\in\cd}$. Let $d_0\in\cd$. There is $k_{01}:d_0\to d_1$ in $\cd$ such that $fF(k_{d_0})=G(k_{d_1})f_0$ for some 
$f_0:FDd_0\to GDd_1$. Continuing this procedure, we get a chain $(k_{ij}:d_i\to d_j)_{i<j<\lambda}$ and morphisms 
$f_i:FDd_i\to FDd_j$ such that 
$$
fF(k_{d_i})=G(k_{d_j})f_i
$$
and
$$
f_jFD(k_{ij})=GD(k_{ij})f_i
$$
for $i<j<\lambda$. In limit steps, we take upper bounds. Now, $A=\colim Dd_i$ and $a=\colim f_i:A\to A$ where $i<\lambda$ implies 
that $a$ is $\lambda^+$-presentable in $\Ins(F,G)$. Since $A$ is in $\ca$, we expressed $K$ as a $\lambda^+$-filtered colimit 
of $\lambda^+$-presentable objects.
\end{proof}

\begin{rem}\label{re3.8}
{\em
(1) Like in \cite{AR}, 2.77, the last three propositions imply that class-accessible categories are closed under lax limits
of strongly class-accessible functors.

(2) Given a class-$\lambda$-accessible category $\ck$, the morphism category $\ck^\to$ is class-$\lambda$-accessible
and both projections $P_1,P_2:\ck^\to\to\ck$ are strongly class-$\lambda$-accessible. Here, $P_1$ sends $A\to B$ to $A$ 
and $P_2$ to $B$. This is analogous to \cite{AR} Ex. 2.c. Since $\ck^\to$ is cocomplete whenever $\ck$ is cocomplete,
$\ck^\to$ is class-locally $\lambda$-presentable provided that $\ck$ is class-locally $\lambda$-presentable. $\lambda$-presentable
objects in $\ck^\to$ are morphisms $A\to B$ such that both $A$ and $B$ are $\lambda$-presentable in $\ck$. Such morphisms will
be called $\lambda$-\textit{presentable}.
}
\end{rem}

\begin{propo}\label{prop3.9}
Let $\ck$ be a class-locally $\lambda$-presentable category and $\ca$ a full subcategory of $\pres_\lambda\ck$.
Then the closure $\tilde{\ca}$ of $\ca$ under colimits in $\ck$ is a class-locally $\lambda$-presentable category.
\end{propo}
\begin{proof}
We know that 
$$
\ck=\Ind_\lambda(\pres_\lambda\ck)
$$
Let $\overline{\ca}$ be the closure of $\ca$ under colimits of less than $\lambda$ objects. Then $\Ind_\lambda(\overline{\ca})$
is class-locally $\lambda$-presentable (by \ref{re2.3} (2)) and, clearly, it is isomorphic to $\tilde{\ca}$.
\end{proof}

\begin{rem}\label{re3.10}
{\em
Moreover, $\tilde{\ca}$ is coreflective in $\ck$. The coreflector is the functor 
$$
\Ind_\lambda(\pres_\lambda\ck)\to\Ind_\lambda(\ca)
$$
induced by the inclusion of $\ca$ into $\pres_\lambda\ck$ (see the proof of \cite{AR}, 1.45).
}
\end{rem}

\section{Weak factorization systems}
An important property of a locally presentable category $\ck$ is that every set $\cc$ of morphisms yields
a weak factorization system $(\cof(\cc),\cc^\square)$ (see \cite{Beke}). We are going to extend this property to class-locally
presentable categories. Recall that, given morphisms $f:A\to B$ and $g:C\to D$ in a category $\ck$, we write
$$
f\square g\quad\quad (f\perp g)
$$
if, in each commutative square
$$
\xymatrix{
A \ar [r]^{u} \ar [d]_{f}& C \ar [d]^{g}\\
B\ar [r]_{v}& D
}
$$
there is a (unique) diagonal $d:B\to C$ with $df=u$ and $gd=v$.

For a class $\cc$ of morphisms of $\ck$ we put
\begin{align*}
\cc^{\square}&=\{g| f\square g \mbox{ for each } f\in \cc\},\\
{}^{\square}\cc&= \{f| f\square g \mbox{ for each } g\in \cc\},\\
\cc^{\perp}&=\{g| f\perp g \mbox{ for each } f\in \cc\},\\
{}^{\perp}\cc&= \{f| f\perp g \mbox{ for each } g\in \cc\}.\\
\end{align*}
The smallest class of morphisms of $\ck$ containing isomorphisms and being closed under transfinite compositions,
pushouts of morphisms from $\cc$ and retracts (in the category $\ck^\to$ of morphisms of $\ck$) is denoted 
by $\cof(\cc)$ while the smallest class of morphisms of $\ck$ closed under all colimits (in $\ck^\to$) and containing 
$\cc$ is denoted as $\colim(\cc)$. Finally, $\Inj(\cc)$ will denote the full subcategory of $\ck$ consisting
of all objects $K$ such that the unique morphism $K\to 1$ to the terminal object of $\ck$ belongs to $\cc^\square$.
These objects $K$ are precisely the objects injective to each morphism $h\in\cc$. Analogously, $\Ort(\cc)$
denotes the full subcategory consisting of objects $K$ such that $K\to 1$ belongs to $\cc^\perp$. These are the objects
orthogonal to each $h\in\cc$.

Given two classes $\cl$ and $\crr$ of morphisms of $\ck$, the pair $(\cl,\crr)$ is called a \textit{weak factorization
system} if
\begin{enumerate}
\item $\crr = \cl^\square$, $\cl = {}^\square \crr$
\end{enumerate}
and
\begin{enumerate}
\item[(2)] any morphism $h$ of $\ck$ has a factorization $h= gf$ with $f\in\cl$ and $g\in\crr$.
\end{enumerate}
The pair $(\cl,\crr)$ is called a \textit{factorization system} if condition (1) is replaced by
\begin{enumerate}
\item[(1')] $\crr = \cl^\perp$, $\cl = {}^\perp \crr$.
\end{enumerate}

\begin{defi}\label{def4.1}
{\em
A class $\cc$ of morphisms of a category $\ck$ is called \textit{cone-coreflective} if, for each morphism $f$ in $\ck$, 
the comma-category $\cc\downarrow f$ has a weakly terminal set.
}
\end{defi}

\begin{rem}\label{re4.2}
{\em
(1) This means that for each $f$ there is a subset $\cc_f$ of $\cc$ such that each morphism $g\to f$ in $\ck^\to$
with $g\in\cc$ factorizes as
$$
g\to h\to f
$$
with $h\in\cc_f$. 

Our terminology is taken from \cite{AR} (the cone $(h\to f)_{h\in\cc_f}$ forms a cone-coreflection from $\cc$ to $f$).
In \cite{Dr}, one calls such classes locally small. This is suggestive because each set of morphisms has this property. 
But this term has an established different meaning in category theory and, moreover, there is no good name for the dual 
concept.  

A union $\cc\cup\cc'$ of two cone-coreflective classes is cone-coreflective because $\cc_f\cup\cc'_f$ is
weakly terminal in $(\cc\cup\cc'\downarrow f)$ for each $f$. In fact, even a union of a set of cone-coreflective
classes is cone-coreflective.

(2) Given a weak factorization system $(\cl,\crr)$, then the class $\cl$ is cone-coreflective. This immediately
follows from the fact that the morphism $f$ in a factorization $h=gf$ of $h$ is a weak coreflection of $h$.
Thus $\cl$ is weakly coreflective, consequently cone-coreflective.

(3) The following  result was proved in \cite{C} where it is called \textit{a generalized small object argument}. 
The proof uses an idea originated in \cite{Dr}. Let $\cc$ be cone-coreflective class of morphisms of a cocomplete category $\ck$. 
Suppose that the domains of all elements in $\cc$ have bounded presentation ranks. Then every morphism $f\in \ck^{\to}$ admits 
a factorization into a $\cc$-cellular morphism (see below) followed by a morphism in $\cc^\square$.

In order to factorize a morphism $f\colon A\to B$, we take a colimit of the diagram  
$$
\xymatrix{
A \\
&&\\
C\ar[uu]^u  \ar@{.}[ur] \ar [rr]_h && D
}
$$
indexed by triples $(u,h,v)$ with $h\in\cc_f$ such that $fu=vh$. This means the pushout of 
$$
\xymatrix@=4pc{
A &\\
\underset{(u,h,v)}{\coprod}C \ar [u]^{<u>} \ar [r]_{\coprod h} & \underset{(u,h,v)}{\coprod}D
}
$$
We obtain a factorization
$$
A=A_0 \xrightarrow{\quad  f_{01}\quad} A_1
             \xrightarrow{\quad f_1\quad} B.
$$
Now, one applies this construction to $f_1$ and continues up to a regular cardinal $\lambda$ such that each
morphism from $\cc$ has a $\lambda$-presentable domain (one takes colimits in limit steps). This implies that $f_\lambda$ is in
$\cc^\square$ and thus
$$
A \xrightarrow{\quad  f_{0\lambda}\quad} A_\lambda
             \xrightarrow{\quad f_\lambda\quad} B
$$
is the desired factorization of $f$. Let us stress that the morphism $f_{0\lambda}$ is a transfinite composition
of pushouts of elements of $\cc$. Such morphisms are called $\cc$-\textit{cellular}.

The resulting factorization of $f$ depends on the subset
$$
\cc^\ast_f=\bigcup\limits_{i<\lambda} \cc_{f_i}.
$$
Given a morphism $(a,b)\colon f\to f'$ where $f\colon A\to B$ and $f'\colon A'\to B'$ in $\ck^\to$, we get the induced morphism 
$a_\lambda:A_\lambda\to A'_\lambda$ with $a_\lambda f_{0\lambda}=f'_{0\lambda} a$ and $f'_\lambda a_\lambda=bf_\lambda$. 
The reason is that, given a triple $(u,h,v)$ for $f$, the composition $(au,bv): h\to f'$ factorizes through a triple $(u',h',v')$
for $f'$. Thus our factorization $\ck^\to\to\ck^\to$ acts both on objects and on morphisms. But we cannot expect that it is functorial, 
i.e., that it preserves composition. The problem is in finding compatible choices of factorizations of triples above.  

There is a canonical choice of a triple above in the case that $\cc^\ast_f\subset\cc^\ast_{f'}$. Thus we can make our factorization functorial 
on each small full subcategory $\ca$ of $\ck^\to$. It suffices to use 
$$
\cc^\ast_\ca =\bigcup\limits_{f\in\ca} \cc^\ast_f
$$
for factorizing. Of course, for different full subcategories $\ca$, the resulting factorizations are different.

(4) Let $\ck$ be class-locally $\mu$-presentable category and $\cc$ be a cone-coreflective class of morphisms whose domains and codomains
are $\lambda$-presentable. Following \ref{re2.8}, we can assume that $\lambda<\mu$. Consider a small full subcategory $\ca$ of $\pres_\mu\ck$.
Since $(\Ind_\mu\ca)^\to=\Ind_\mu(\ca^\to)$, 
$$
\cc_\ca =\bigcup\limits_{f\in\ca} \cc_f
$$
can be used as $\cc_h$ for each morphism $h$ in $\Ind_\mu\ca$. Since pushouts commute with $\mu$-filtered colimits in $\ck$, 
$\cc^\ast_\ca$ can be used as $\cc^\ast_h$. Thus we can make our factorization functorial on $\Ind_\mu\ca$. Moreover,
the corresponding functor is $\mu$-accessible. Thus it is strongly $\nu$-accessible for some $\nu$ but this cardinal depends on $\ca$ 
in general. Like in (3), the factorization itself depends on $\ca$.

In the case when $\ck$ is locally presentable and $\cc$ is a set, we have $\ck=\Ind_\mu\ca$ for some $\mu$ and we get a strongly accessible 
functorial factorization on $\ck$. This was claimed by J. H. Smith and our proof completes those from \cite{D}, 7.1 and \cite{R1} 3.1 (here, 
one should use our triples $(u,h,v)$ instead of pairs $(u,h)$).
}
\end{rem}

We have proved the following theorem.

\begin{theo}\label{th4.3} 
Let $\ck$ be a class-locally presentable category, $\cc$ a cone-coreflective class of morphisms of $\ck$ and
assume that there is a regular cardinal $\lambda$ such that each morphism from $\cc$ has the $\lambda$-presentable 
domain. Then $(\cof(\cc),\cc^\square)$ is a weak factorization system in $\ck$.
\end{theo}

A full subcategory $\cl$ of $\ck$ is called \textit{weakly reflective} in $\ck$ if, for each $K$ in $\ck$,
the comma-category $K\downarrow\cl$ has a weakly initial object. It means the existence of a morphism
$r:K\to K^\ast$ with $K^\ast\in\cl$ such that each morphism $K\to L$ with $L\in\cl$ factorizes through $r$.

\begin{coro}\label{cor4.4}
Let $\ck$ be a class-locally presentable category, $\cc$ a cone-coreflective class of $\lambda$-presentable morphisms of $\ck$.  
Then $\Inj(\cc)$ is weakly reflective and closed under $\lambda$-filtered colimits in $\ck$.
\end{coro}
\begin{proof}
A weak reflection of $K$ is given by a $(\cof(\cc),\cc^\square)$ factorization
$$
K \xrightarrow{\quad  r\quad} K^\ast
             \xrightarrow{\quad \quad} 1.
$$
Assume that $(k_d:K_d\to K)_{d\in\cd}$ is a $\lambda$-filtered colimit of objects $K_d\in\Inj(\cc)$. Given
$h:C\to D$ in $\cc$ and $u:C\to K$, there is $d\in\cd$ and $u':C\to K_d$ with $u=k_du'$. There is $v:D\to K_d$
with $vh=u'$. Since $k_dvh=u$, $K\in\Inj(\cc)$.
\end{proof}

\begin{coro}\label{cor4.5}
Let $\ck$ be a class-locally presentable category, $\cc$ a cone-coreflective class of morphisms of $\ck$. Let $\lambda$ be a regular cardinal 
such that each morphism from $\cc$ is $\lambda$-presentable. Then $\cc^\square$ is closed under $\lambda$-filtered colimits in $\ck^\to$.
\end{coro}
\begin{proof}
It suffices to observe that $g$ has a right lifting property w.r.t. $f:A\to B$ if and only if $g$ is injective 
in $\ck^\to$ to the morphism
$$
(f,\id_B):f\to\id_B.
$$
Since $f$ is $\lambda$-presentable in $\ck^\to$, the result follows from \ref{cor4.4}. 
\end{proof}

The following example shows that, in \ref{cor4.4}, $\cc^\square$ and $\Inj(\cc)$ do not need to be class-accessible.

\begin{exam}\label{ex4.6}
{\em
Let $\co$ be the category whose objects are ordinal numbers (considered as well-ordered sets) and isotone maps.
The category $\cp(\co)$ from \ref{ex2.2} (2) can be understood as a transfinite extension of simplicial sets because $\hom(-,\alpha +1)$
can be taken as the $\alpha$-simplex $\Delta_\alpha$. Since $\co$ is approximately complete, $\cp(\co)$ is class-locally presentable.

Let $\Delta_{1s}$ be the symmetric $1$-simplex. This means the object having two points $0$ and $1$, two non-degenerated $1$-simplices $[0,1]$ 
and $[1,0]$ and all $\alpha$-simplices made out from these. In more detail, $\Delta_{1s}$ is a coequalizer of morphisms $f,g:\Delta_1\to\Delta_2$ 
where $f$ sends $\Delta_1$ to the face $[0,2]$ of $\Delta_2$ and $g$ is the constant morphism on $0$. Let $j:\Delta_1\to\Delta_{1s}$ 
be the inclusion on $[0,1]$. Then, following \ref{th4.3}, $(\cof(j),j^\square)$ is a weak factorization system cofibrantly generated by a single 
morphism. Injectivity with respect to $j$ means that each $1$-simplex is symmetrized. We will show that $\Inj(j)$ is not class-accessible. 
The reason is that, given a regular cardinal $\alpha$, each weak reflection of $\Delta_\alpha$ to $\Inj(j)$ adds at least 
$\alpha$-many $1$-simplices because we have to symmetrize each $1$-simplex in $\Delta_\alpha$. Let $(\Delta_\alpha)^\ast$ denote 
the weak reflection which adds to each $1$-simplex just one symmetric $1$-simplex. We will show that, for each regular cardinal 
$\alpha$, $(\Delta_\alpha)^\ast$ is not $\alpha$-presentable in $\Inj(j)$.  

Let $\alpha$ be a cardinal and $(\Delta_\alpha)^{\ast\beta}$, $\beta\leq\alpha$ extends each $1$-simplex 
in $\Delta_\alpha$ by $\beta$-many symmetric $1$-simplices indexed by ordinals $i<\beta$. Then
$(\Delta_\alpha)^{\ast\alpha}$ is an $\alpha$-directed colimit of $(\Delta_\alpha)^{\ast\beta}$, 
$\beta<\alpha$ in $\Inj(j)$. We index all $1$-simplices in $\Delta_\alpha$ by ordinals $i<\alpha$ and
consider the morphism $f:\Delta_\alpha^\ast\to (\Delta_\alpha)^{\ast\alpha}$ sending the symmetric $1$-simplex 
added to the $i$-th $1$-simplex $e_i$ in $\Delta_\alpha$ to the $i$-th added symmetric $1$-simplex extending $e_i$.
Clearly, $f$ does not factorize through any $(\Delta_\alpha)^{\ast\beta}$, $\beta<\alpha$.

Now, assume that $\Inj(j)$ is class-$\lambda$-accessible. Then $\Delta_\lambda^\ast$ is a $\lambda$-directed
colimit $k_d:K_d\to\Delta_\lambda^\ast$ of $\lambda$-presentable objects $K_d$ in $\Inj(j)$. Since $\Inj(j)$ 
is closed under filtered colimits in $\cp(\co)$, the weak reflection $r:\Delta_\lambda\to\Delta_\lambda^\ast$ 
factorizes through some $k_d$, i.e., $r=k_df$. Since $r$ is a weak reflection, there exists 
$g:\Delta_\lambda^\ast\to K_d$ with $gr=f$. Thus $k_dgr=r$. Consider a non-degenerated $1$-simplex $[i,j]$ in $\Delta_\lambda$.
Then, $\Delta_\lambda^\ast$ contains the new $1$-simplex $[j,i]$ such that $[i,j,i]$ is a $2$-simplex with $[i,i]$ degenerated.
Since $[j,i]$ is the only $1$-simplex with this property, $k_dg$ must send it to itself. Thus $k_dg=\id_{\Delta_\lambda^\ast}$. 
Therefore $\Delta_\lambda^\ast$ is $\lambda$-presentable as a retract of $K_d$, which is a contradiction. Consequently, $\Inj(j)$ 
is not class-accessible. 
}
\end{exam}

Recall that a morphism is called $\lambda$-presentable if it has the $\lambda$-pre\-sen\-tab\-le domain and the $\lambda$-presentable 
codomain (see \ref{re3.8} (2)).

\begin{defi}\label{def4.7}
{
\em
Let $\ck$ be a class-locally $\lambda$-presentable category. A weak factorization system $(\cl,\crr)$ in $\ck$ is called
\textit{cofibrantly class}-$\lambda$-\textit{generated} if $\cl=\cof(\cc)$ for a cone-coreflective class $\cc$ of morphisms such that
\begin{enumerate}
\item[(1)] morphisms from $\cc$ are $\lambda$-presentable and
\item[(2)] any morphism between $\lambda$-presentable objects has a weak factorization with the middle object $\lambda$-presentable.
\end{enumerate}

A cone-coreflective class $\cc$ of morphisms will be called $\lambda$-\textit{bounded} if $(\cof(\cc),\cof(\cc)^\square)$  
satisfies conditions (1) and (2) above.
  
We say that $(\cl,\crr)$ is \textit{cofibrantly class-generated} if there is a regular cardinal $\lambda$ such that $(\cl,\crr)$ 
is cofibrantly class-$\lambda$-generated. The same for \textit{bounded}. 
}
\end{defi}

\begin{rem}\label{re4.8}
{
\em
(1) Any set $\cc$ is a cone-coreflective class. In this case, the factorization is always functorial (because, in \ref{re4.2}, $\cc_f=\cc$
for each $f$). But, \ref{ex4.6} shows that $\cc$ does not need to be bounded. This follows from \ref{cor4.10} but there is a direct
verification. Assume that $\Delta_\lambda\to\Delta_0$ has a weak factorization with the $\lambda$-presentable middle object $K$. There
are morphisms $k:K\to \Delta_\lambda^\ast$ and $g:\Delta_\lambda^\ast\to K$ and, like in \ref{ex4.6}, $kg=\id_{\Delta_\lambda^\ast}$.
Thus $\Delta_\lambda^\ast$ is $\lambda$-presentable in $\cp(\co)$ and, therefore, in $\Inj(j)$, which cannot happen.

(2) Let $\cc$ be a cone-coreflective class of morphisms in a class-locally $\lambda$-presentable category. Given $f:A\to B$ and $\cc_f$, 
we denote by $\ct_f$ a set of all triples $(u,h,v)$ from \ref{re4.2} (3) with $h\in\cc_f$ which are needed for cone-coreflectivity. Let
$$
\ct^\ast_f=\bigcup\limits_{i<\lambda} \ct_{f_i}.
$$
Assume the existence of a regular cardinal $\mu > \lambda$ such that that the cardinality of $\ct_f^\ast$ is smaller than $\mu$ 
for each morphism $f\colon A\to B$ with $A$ and $B$ $\mu$-presentable. Then for such an $f$, all objects $A_i$, $i\leq\lambda$ 
from \ref{re4.2}(3) are $\mu$-presentable and $A_\lambda$ is the middle object in a weak factorization of $f$. Thus $\cc$ is $\mu$-bounded. 
Due to choices of sets $\ct_f$, we cannot expect to get a functorial factorization in this way. Even, we cannot make it functorial
on small full subcategories. 

(3) Given a weak factorization $f=f_2f_1$ from \ref{def4.7} (2), we can choose $\cc_f=\{f_1\}$ and $\ct_f=\{(\id_A,f_1,f_2)\}$.
Then (2) above implies that \ref{re4.2} (2) yields a weak factorization from \ref{def4.7} (2). 

(4) We do not know whether a $\lambda$-bounded $\cc$ is $\mu$-bounded for $\lambda\vartriangleleft\mu$. This is true provided that
the factorization is functorial -- then it is given by a strongly class-$\lambda$-accessible functor which is strongly class-$\mu$-accessible
(see \ref{re2.8}). 

(5) We say that $\cc$ is $(\lambda,\lambda^+)$-\textit{bounded} if it satisfies \ref{def4.7} (1) for $\lambda$ and (2) for $\lambda^+$.
Each $(\lambda,\lambda^+)$-bounded class is $\lambda^+$-bounded. We will show that a union $\cc\cup\cc'$ of two $(\lambda,\lambda^+)$-bounded 
classes is $\lambda^+$-bounded. 

The union is cone-coreflective (see \ref{re4.2} (1)). Let $f:A\to B$ be a morphism between $\lambda^+$-presentable objects. We proceed
by a transfinite construction for $i\leq\lambda$. We start with its $(\cof(\cc),\cc^\square)$ factorization
$$
A_0=A\xrightarrow{\quad  f_{01}\quad} A_1
             \xrightarrow{\quad g_0\quad} B
$$
Then we take a $(\cof(\cc'),(\cc')^\square)$ factorization of $g_0$
$$
A_1\xrightarrow{\quad  f_{12}\quad} A_2
             \xrightarrow{\quad g_1\quad} B
$$
We put $f_{02}=f_{12}f_{01}$ and continue the procedure by a $(\cof(\cc),\cc^\square)$ factorization of $g_1$. In a limit step
$i$, $f_{ji}:A_j\to A_i$ is given by a transfinite composition and $g_i:A_i\to B$ is the induced morphism. We finish at $\lambda$
and get a factorization
$$
f:A\xrightarrow{\quad  f_{0\lambda}\quad} A_\lambda
             \xrightarrow{\quad g_\lambda\quad} B
$$
The object $A_\lambda$ is $\lambda^+$-presentable because $\lambda^+$-presentable objects are closed under $\lambda^+$-small colimits
and both $\cc$ and $\cc'$ satisfy (2) for $\lambda^+$. We have $f_{0\lambda}\in\cof(\cc\cup\cc')$. Finally, $g_\lambda$ is both 
a $\lambda$-filtered colimit of morphisms from $\cc^\square$ and a $\lambda$-filtered colimit of morphisms from $(\cc')^\square$.
Following \ref{cor4.5}, $g_{\lambda}\in \cc^\square\cap(\cc')^\square=(\cc\cup\cc')^\square$.  
}
\end{rem}

\begin{theo}\label{th4.9}
Let $\ck$ be a class-locally $\lambda$-presentable category and $(\cl,\crr)$ a cofibrantly class-$\lambda$-generated weak factorization 
system. Then $\crr$ is a class-$\lambda$-accessible category strongly $\lambda$-accessibly embedded in $\ck^\to$.
\end{theo}
\begin{proof}
Following \ref{re3.8} (2), $\ck^\to$ is class-locally $\lambda$-presentable and the projections $P_1,P_2\colon \ck^\to\to\ck$ 
are strongly class-$\lambda$-accessible. Following \ref{cor4.5}, $\crr$ is closed in $\ck^\to$ under 
$\lambda$-filtered colimits. Consider a morphism $f \colon K\to L$ in $\crr$. We have to show that $f$ is a $\lambda$-filtered 
colimit of $\lambda$-presentable objects from $\crr$.
Following \ref{re2.3} (2), $\ck^\to$ is a union of a chain
$$
\Ind_\lambda(\ca_0^\to)\subseteq\Ind_\lambda(\ca_1^\to)\subseteq\dots\Ind_\lambda(\ca_i^\to)\subseteq\dots
$$
of locally $\lambda$-presentable categories and strongly class-$\lambda$-accessible functors. There is $i_0$ such that $f$ belongs to
$\Ind_\lambda(\ca_{i_0}^\to)$. Following \ref{cor4.4} and \ref{def4.7} (2), given a $\lambda$-presentable object $h$ of $\ck^\to$ ,
its weak factorization $h_2h_1$ has $h_2$ $\lambda$-presentable in $\ck^\to$ and thus in $\crr$. Since $\ca_{i_0}^\to$ is small, there 
is $i_1$ such that $\ca_{i_1}^\to$ contains all $h_2$ for $h$ from $\ca_{i_0}^\to$. Analogously, there is $i_2$ such that $\ca_{i_2}^\to$ 
contains all $h_2$ for $h$ from $\ca_{i_1}^\to$. We will continue this procedure, in limit steps we take $i_j$ as the supremum of all $i_k$ 
with $k<j$. Let $\cb$ be the intersection of $\ca_{i_\lambda}^\to$ and $\crr$. Consider a morphism $(a,b):h\to f$ with $h:A\to B$ 
in $\ca_{i_\lambda}^\to$. Then $h$ belongs to $\ca_{i_j}^\to$ for some $j<\lambda$. Thus $\cb$ contains $h_2$. Since $f$ is 
in $\crr$ and $h_1$ in $\cl$, there is $c:C\to K$ such that $fc=bh_2$ and $ch_1=a$; here $C$ is the middle object in the weak
factorization $h=h_2h_1$. Hence $(a,b)$ factorizes through $h_2$. It remains to show that $\cb\downarrow f$ is $\lambda$-filtered. 
Then it will be cofinal in $\ca_{i_\lambda}^\to\downarrow f$ and $f$ will be the canonical colimit of $\cb$-objects.

Let $\cx$ be a $\lambda$-small subcategory of $\cb\downarrow f$. There is $j<\lambda$ such that $\cx$ is a subcategory 
of $\ca_{i_j}\downarrow f$. Thus $\cx$ has an upper bound $(a_X,b_X):h\to f$, $h\in\cx$ in $\ca_{i_j}^\to\downarrow f$. 
Then $h_2$ is an upper bound of $\cx$ in $\cb\downarrow f$. This proves that the latter category is $\lambda$-filtered.
\end{proof}

\begin{coro}\label{cor4.10}
Let $\ck$ be a class-locally $\lambda$-presentable category with a $\lambda$-presentable terminal object and $(\cl,\crr)$ be 
a cofibrantly class-$\lambda$-generated weak factorization system. Then $\Inj(\cl)$ is a class-accessible category strongly accessibly 
embedded in $\ck$.
\end{coro}
\begin{proof}
Since a weak reflection of an object $K$ in $\Inj(\cl)$ is given by a weak factorization of its morphism $K\to 1$ to a terminal object, 
the result follows from the proof of \ref{th4.9}.
\end{proof}
 
\begin{theo}\label{th4.11}
Let $\ck$ be a class-locally presentable category, $\cc$ a cone-coreflective class of morphisms of $\ck$ and
assume that there is a regular cardinal $\lambda$ such that each morphism from $\cc$ is $\lambda$-presentable.
Then $(\colim(\cc),\cc^\perp)$ is a factorization system in $\ck$.
\end{theo}
\begin{proof}
Given a morphism $f:A\to B$, we form the pushout of $f$ and $f$ and denote by $f^\ast$ a unique morphism making the following 
diagram commutative
$$
\xymatrix@C=4pc@R=4pc{
A\ar[r]^f \ar[d]_f &
B\ar[d]^{p_2}\ar[ddr]^{\id_B}&\\
B\ar[r]_{p_1}\ar[drr]_{\id_B}& A^\ast \ar[dr]^{f^\ast}&\\
&&B
}
$$
Let $\bar{\cc}=\cc\cup\cc^\ast$ where $\cc^\ast=\{f^\ast \,|\, f\in\cc\}$. Since pushouts of $\lambda$-presentable
objects are $\lambda$-presentable, each morphism from $\bar{\cc}$ has a $\lambda$-presentable domain. It suffices to show 
that the class $\bar{\cc}$ is cone-coreflective. Then, following \ref{th4.3}, $(\cof(\bar{\cc}),\bar{\cc}^\square)$ is 
a weak factorization system. Following the proof of 4.1 in \cite{Bo}, we conclude that $\bar{\cc}^\square=\cc^\perp$ and
$(\cof(\bar{\cc}),\cc^\perp)$ is a factorization system. Finally, the proof of 2.2 in \cite{FR} shows that 
$\cof(\bar{\cc})=\colim(\cc)$.

We have to prove that the class $\cc^\ast$ is cone-coreflective. Let $g$ be a morphism in $\ck$ and take the pullback
$$
\xymatrix@C=3pc@R=3pc{
D_\ast \ar [r]^{q_1} \ar [d]_{q_2}& C \ar [d]^{g}\\
C\ar [r]_{g}& D
}
$$
We have a unique $g_\ast:C\to D_\ast$ such that $q_ig_\ast=\id_C$ for $i=1,2$.
We will show that commutative squares
$$
\xymatrix@C=3pc@R=3pc{
A^\ast \ar [r]^{u} \ar [d]_{f^\ast}& C \ar [d]^{g}\\
B\ar [r]_{v}& D
}
$$
with $f\in\cc$ uniquely correspond to commutative squares
$$
\xymatrix@C=3pc@R=3pc{
A \ar [r]^{t} \ar [d]_{f}& C \ar [d]^{g_\ast}\\
B\ar [r]_{h}& D_\ast
}
$$
The correspondence assigns to $u$ and $v$ the pair $t,h$ where $t=up_1f$ and $h$ is given by $q_ih=up_i$ for $i=1,2$ (because $gup_1=v=gup_2$). 
Since $q_ihf=up_if=t=q_ig_\ast t$ for $i=1,2$, we have $hf=g_\ast t$. Conversely, given $t$ and $h$, we get $u$ and $v$ by means of $up_i=q_ih$ 
for $i=1,2$ and $v=gq_1h$. Since $gup_i=gq_ih=v=vf^\ast p_i$ for $i=1,2$, we have $gu=vf^\ast$. The one-to-one correspondence between $u$ and $h$ 
is evident. Since
$$
v=vf^\ast p_1=gup_1=gq_1h
$$
and
$$
t=q_1g_\ast t=q_1hf=up_1f,
$$
there is a one-to-one correspondence between $v$ and $t$ as well.
 
Since the class $\cc$ is cone-coreflective, the comma-category $\cc\downarrow g_\ast$ has a weakly terminal set 
$\cc_{g_\ast}$. We will show that the corresponding set $(\cc_{g_\ast})^\ast$ is weakly terminal in $\cc^\ast\downarrow g$.
Given $(u,v):f^\ast\to g$, we take the corresponding $(t,h):f\to g_\ast$. There is a factorization
$$
(t,h):f\xrightarrow{\quad  (t_1,h_1)\quad} e
             \xrightarrow{\quad (t_2,h_2)\quad} g_\ast
$$
with $e:X\to Y$ in $\cc_{g_\ast}$. Let $(u_2,v_2):e^\ast\to g$ corresponds to $(t_2,h_2)$, $v_1=h_1$ and $u_1:A^\ast\to X^\ast$ is
induced by $u_1p_i=\overline{p}_ih_1$ for $i=1,2$ where $\overline{p}_i$ are from the pushout defining $e^\ast$.  Since
$$
u_2u_1p_i=u_2\overline{p}_ih_1=q_ih_2h_1=q_ih=up_i
$$
for $i=1,2$, we have $u_2u_1=u$. Further,
$$
v_2v_1=(gq_1h_2)h_1=gq_1h=v.
$$
Finally, since 
$$
e^\ast u_1p_i=e^\ast\overline{p}_ih_1=h_1=v_1=v_1f^\ast p_i
$$
for $i=1,2$, we have $d^\ast u_1= v_1f^\ast$. Hence we get a factorization
$$
(u,v):f^\ast\xrightarrow{\quad  (u_1,v_1)\quad} e^\ast
             \xrightarrow{\quad (u_2,v_2)\quad} g
$$
with $e^\ast\in(\cc_{g_\ast})^\ast$. 
\end{proof} 

\begin{coro}\label{cor4.12}
Let $\ck$ be a class-locally $\lambda$-presentable category and $\cc$ a cone-coreflective class of $\lambda$-presentable morphisms of $\ck$.
Then $\Ort(\cc)$ is reflective and closed under $\lambda$-filtered colimits in $\ck$. Moreover, $\Ort(\cc)$ is class-locally $\lambda$-presentable.  
\end{coro}
\begin{proof}
A reflection of $R(K)$ is given by a $(\colim(\cc),\cc^\perp)$ factorization
$$
K \xrightarrow{\quad  r\quad} R(K)
             \xrightarrow{\quad \quad} 1.
$$
Since $\Ort(\cc)=\Inj(\bar{\cc})$ where $\bar{\cc}$ is from the proof of \ref{th4.11}, $\Ort(\cc)$ is closed 
under $\lambda$-filtered colimits in $\ck$ (following \ref{cor4.4}). Following \ref{th2.6}, $\Ort(\cc)$ is class-locally 
$\lambda$-presentable.  
\end{proof}

\begin{rem}\label{re4.13}
{
\em
(1) Given $\cc$ from \ref{cor4.12}, we show in the same way as in \ref{th4.9} that $\cc^\perp$ is class-locally $\lambda$-presentable.

Concerning $\colim(\cc)$, we know that it is coreflective in $\ck^\to$ and the coreflector 
$$
R:\ck^\to\to\colim(\cc)
$$ 
preserves $\lambda$-filtered colimits. Hence $\colim(\cc)$ is a full image of a class-accessible functor. But we do not
know whether $R:\ck^\to\to\ck^\to$ is strongly class-accessible and we thus do not know whether $\colim(\cc)$ is class-locally
presentable. What is missing is the condition \ref{def4.7} (2).

(2) Let $\ck$ be a class-locally $\lambda$-presentable category written as a union of a chain of locally 
$\lambda$-presentable full strongly $\lambda$-accessibly embedded subcategories $\ck_i$ (see \ref{re3.10}). Let
$\cc$ be a class-$\lambda$-accessible full subcategory of $\ck$ such that $(\cc\cap\ck_i)^\perp$ 
(calculated in $\ck_i$) is reflective in $\ck_i$. Then $\cc^\perp$ is reflective in $\ck$. This follows from the fact
that, given a $\lambda$-filtered colimit $f=\colim (f_d)_{d\in\cd}$, then
$$
\bigcap\limits_{d\in\cd} f_d^\perp\subseteq f^\perp
$$
 
This observation is behind the proof that each small scheme has a sheafification with respect to the flat topology
(see \cite{W}, small functors are called basically bounded there and 3.1 shows that faithfully flat morphisms 
are class-finitely-accessible). It would be interesting to know whether this is true for the etale topology as well.
}
\end{rem}

\end{document}